\documentclass[12pt]{article}

\usepackage{amsmath}
\usepackage{amssymb,latexsym,xcolor,dsfont,enumitem,mathrsfs,stmaryrd,bm,mathtools} 
\usepackage{cite}
\usepackage[colorlinks=true,linkcolor=blue!50!black,pagecolor=blue!50!black,citecolor=green!50!black]{hyperref}
\usepackage[auth-sc]{authblk}
\usepackage{multirow}
\usepackage{longtable}

\usepackage{tikz}
\usepackage{pgf}
\usetikzlibrary{calc}
\usetikzlibrary{intersections}
\usetikzlibrary{fadings}

\usepackage[cp1251]{inputenc}
\usepackage[T1]{fontenc}
\usepackage[russian,english]{babel}

\makeatletter
\def\thmstyle{\it} 
\def\@begintheorem#1#2{\it \trivlist \item[\hskip
        \labelsep{\bf #1\ #2.}]\thmstyle}
\def\@opargbegintheorem#1#2#3{\it \trivlist \item[\hskip
        \labelsep{\bf #1\ #2\ (#3).}]\thmstyle}
\makeatother

\newtheorem{theorem}{{\indent}Theorem} 

\definecolor{dgreen}{RGB}{0,150,0}

\begin{document}

\title{Complexity of the variable-free fragments\protect\\ 
of non-normal modal logics\protect\\
(extended version)\footnote{The research is supported by the MSHE ``Priority 2030'' strategic academic leadership program. The paper is an extended version of the abstracts submitted to the conference Smirnov Readings~\cite{KudinovRybakov:2025:SR}.}\\~}

\author[1]{A.\,Kudinov}
\author[2]{M.\,Rybakov}
\affil[1]{HSM MIPT, HSE University}
\affil[2]{HSM MIPT, HSE University, Tver State University}

\date{}

\newcommand{\logic}[1]{\mathbf{#1}}
\newcommand{\num}[1]{\mathds{#1}}
\newcommand{\numbers}         [1]       {\mathds{#1}}
\newcommand{\numN}                      {\numbers{N}}
\newcommand{\numNp}                     {\numbers{N}^+}
\newcommand{\numQ}                      {\numbers{Q}}
\newcommand{\numR}                      {\numbers{R}}
\newcommand{\numC}                      {\numbers{C}}
\newcommand{\numA}                      {\numbers{A}}
\newcommand{\numAbar}                   {\bar{\numbers{A}}}
\newcommand{\ccls}[1]{\mathrm{#1}}
\newcommand{\HideText}[1]{}

\maketitle


\begin{abstract}
We show that the satisfiability problem for the variable-free fragment of every modal logic containing classical propositional logic and contained in the weak Grzegorczyk logic is $\ccls{NP}$-hard. In particular, the variable-free fragments of the non-normal modal logics $\logic{E}$, $\logic{EM}$, $\logic{EN}$, and $\logic{EMN}$ are $\ccls{coNP}$-complete.
\end{abstract}

Modal logic is widely used in philosophy to analyze and formalize various categories related to possibility, necessity, knowledge, obligations, time, and other concepts. When interpreting knowledge, logics between $\logic{S4}$ and $\logic{S5}$ are often used, but sometimes weaker systems are considered, down to the minimal normal modal logic $\logic{K}$. For interpreting ``belief'', logics without the axiom $\Box p \to p$ are frequently employed. All these modal logics are normal, and they imply an agent's ability to perform arbitrarily long and complex logical deductions. Indeed, normality, i.e., the presence of the formula $\Box (p\to q) \to (\Box p \to \Box q)$, essentially means the agent's knowledge is closed under modus ponens. 
When combined with the necessitation rule and the classical tautologies, this implies that the agents know (or believe in) all logical consequences of their knowledge.
The study of knowledge and belief for agents without ``logical omniscience'' necessarily leads us to non-normal modal logics. A good overview of the use of non-normal modal logics in philosophy can be found in~\cite{Pacuit:2017}.

A crucial property of a formal logical system is its computational complexity. Low complexity enables efficient use of the system in various practical applications analyzing agents' beliefs and/or knowledge. Studying fragments with a finite set of variables provides better insight into how this complexity is structured and the potential feasibility (or impossibility) of creating more efficient algorithms.

Many modal and superintuitionistic propositional logics are $\ccls{PSPACE}$-hard~\cite{Ladner77,Statman1979,Chagrov:1985}, and often $\ccls{PSPACE}$-hardness can be proven for their two-vari\-able fragments, one-variable fragments or even variable-free fragments~\cite{Spaan:1993:PhD,Halpern:1995,Svejdar:2003,ChR:2003:AiML,MR:2002,MR:2003,MR:2004,MR:2006,MR:2007}. However, when adding axioms that limit the height or width of Kripke frames, we obtain ``strong'' logics with $\ccls{NP}$-complete satisfiability problems. For example, the logics $\logic{K5}$, $\logic{KD45}$, $\logic{S5}$, $\logic{K4.3}$, $\logic{S4.3}$, $\logic{GL.3}$, $\logic{Grz.3}$, and many others are $\ccls{coNP}$-complete (see, e.g.,~\cite[Chapter~18]{ChZ}). A~similar situation arises when the logic is ``weak''. As an example, consider the non-normal modal logic $\logic{E}$, obtained by closing the classical logic $\logic{Cl}$ under substitution and the rule
$$
\frac{A\leftrightarrow B}{\Box A\leftrightarrow \Box B}.
$$
The logic $\logic{E}$ is contained in the minimal normal modal logic~$\logic{K}$, yet its satisfiability problem is the same as for $\logic{Cl}$, i.e., $\ccls{NP}$-complete~\cite{Vardi:1989}. 
The following axioms are often added to~$\logic{E}$:
$$
\begin{array}{lcl}
    \bm{am} & = & \Box (p \land q) \to \Box p;\\
    \bm{ac} & = & \Box p \land \Box q \to \Box (p \land q);\\
    \bm{an} & = & \Box \top.
\end{array}
$$
The nomenclature of non-normal logics is typically constructed by appending to the logic's name the second letters of the added axioms. For example, $\logic{EMN} = \logic{E} + \bm{am} + \bm{an}$. Logics containing the axiom $\bm{am}$ are called monotonic, and the logic $\logic{EMCN}$ coincides with the minimal normal modal logic $\logic{K}$ (see \cite[Section 2.3]{Pacuit:2017}).

It is known that the logics $\logic{K5}$, $\logic{KD45}$, and $\logic{S5}$, while $\ccls{coNP}$-complete in a language with countably many variables, are polynomially decidable in any language with a fixed finite number of variables, which follows from their local tabularity; the same applies to other locally tabular logics or their locally tabular fragments (e.g., the disjunction-free fragment of intuitionistic logic~\cite{Diego}). At the same time, the logics $\logic{S4.3}$ and $\logic{Grz.3}$ remain $\ccls{coNP}$-hard in a two-variable language, while $\logic{K4.3}$ and $\logic{GL.3}$ are hard even in a one-variable language~\cite{ChR:2003:AiML}.

A natural question arises: what happens to the complexity of ``weak'' logics when their language is restricted to a finite set of variables? In particular, what is the complexity of such fragments for the aforementioned logics?

The work \cite{Vardi:1989} shows that not only $\logic{E}$ but also $\logic{EM}$, $\logic{EN}$, and $\logic{EMN}$ are $\ccls{coNP}$-complete in a language with countably many variables.
Note that $\logic{E}$, $\logic{EM}$, $\logic{EN}$, and $\logic{EMN}$ are not locally tabular: the formulas of the set $\{\Box^k\bot:k\in\numN\}$ are pairwise non-equivalent in each of them (it suffices to observe that they are all contained in~$\logic{K}$, where these formulas are pairwise non-equivalent). Using methods from~\cite{ChR:2003:AiML} and generalizing results from~\cite{ChR:2003:AiML}, the following theorems can be easily proved.

\begin{theorem}
Let $L$ be a modal logic containing $\logic{Cl}$ and contained in $\logic{Grz.3}$. Then the $L$-satisfiability problem is $\ccls{NP}$-hard for the two-variable fragment of $L$.
\end{theorem}

\begin{proof}
Let $p$ and $q$ be propositional variables. Define the following formulas recursively:
$$
\begin{array}{lcl}
\delta_0 & = & \Box q; \\
\delta_{n+1} & = & q\wedge \Diamond (\neg q \wedge \Diamond \delta_n). \\
\end{array}
$$
For each $n\in\numN$, let 
$$
\begin{array}{lcl}
\alpha_n & = & \delta_n\wedge \neg\delta_{n+1}; \\
\beta_n & = & \Diamond (p \wedge \alpha_n). \\
\end{array}
$$
Let $\mathcal{P}=\{p_k : k\in\numN\}$ be a countable set of propositional variables in the language (not necessarily containing all of them). Let $\bm{s}$ be the substitution defined for each variable $r$ as follows:
$$
\begin{array}{lcl}
\bm{s}r & = & 
 \left\{
 \begin{array}{rl}
 \beta_k, & \mbox{if $r=p_k$;} \\
       r, & \mbox{if $r\not\in\mathcal{P}$.} 
 \end{array}
 \right.
\end{array}
$$
It is easy to see that for every classical formula $\varphi$ with variables $p_1,\ldots,p_n$, the following equivalence holds:
$$
\begin{array}{lcl}
\varphi\in\logic{Cl} & \iff & \bm{s}\varphi \in L. \\
\end{array}
\eqno{\mbox{$({\ast})$}}
$$
The implication $({\Rightarrow})$ in the equivalence~\mbox{$({\ast})$} holds because $\logic{Cl}\subseteq L$ and the logic $L$ is closed under substitution.
To justify the implication $({\Leftarrow})$ in the equivalence~\mbox{$({\ast})$}, suppose that $\varphi\not\in\logic{Cl}$. 
Since $\varphi\not\in\logic{Cl}$, there exists a valuation $v\colon\mathcal{P}\to\{0,1\}$ that refutes the formula $\varphi$. Consider the Kripke frame $F = \langle \omega \cup \{\omega\}, \geqslant\rangle$ (i.e., $F$ is $1+\omega^\ast$ with the non-strict order\footnote{In fact, we may not differ $\omega$ from~$\numN$.}), and take a valuation such that for each $k\in\omega$, the following conditions hold:
$$
\begin{array}{lcl}
k\models q & \iff & \mbox{$k$ is even;} \\
k\models p & \iff & \mbox{$k$ is even and $v(p_k)=1$.} \\
\end{array}
$$
Then it is easy to see that for each $k\in\numN$,
$$
\begin{array}{lcl}
\omega\models \beta_k & \iff & v(p_k) = 1, \\
\end{array}
$$
which implies that $\omega\not\models\bm{s}\varphi$. It remains to note that $F\models\logic{Grz.3}$, so $\bm{s}\varphi\not\in\logic{Grz.3}$, and hence $\bm{s}\varphi\not\in L$, since $L\subseteq\logic{Grz.3}$.
\end{proof}

\begin{theorem}
Let $L$ be a modal logic containing $\logic{Cl}$ and contained in $\logic{GL.3}$. Then the $L$-satisfiability problem is $\ccls{NP}$-hard for the one-variable fragment of $L$.
\end{theorem}

\begin{proof}
Let $p$ be a propositional variable. Define the following formulas recursively:
$$
\begin{array}{lcl}
\delta'_0 & = & \Box \bot; \\
\delta'_{n+1} & = & \Diamond \delta_n. \\
\end{array}
$$
For each $n\in\numN$, let 
$$
\begin{array}{lcl}
\alpha'_n & = & \delta'_n\wedge \neg\delta'_{n+1}; \\
\beta'_n & = & \Diamond (p \wedge \alpha'_n). \\
\end{array}
$$
Let $\mathcal{P}=\{p_k : k\in\numN\}$ be a set of propositional variables, and let $\bm{s}'$ be the substitution defined for each variable $r$ as follows:
$$
\begin{array}{lcl}
\bm{s}'r & = & 
 \left\{
 \begin{array}{rl}
 \beta'_k, & \mbox{if $r=p_k$;} \\
       r, & \mbox{if $r\not\in\mathcal{P}$.} 
 \end{array}
 \right.
\end{array}
$$
Then, for every classical formula $\varphi$ with variables $p_1,\ldots,p_n$, the following equivalence holds:
$$
\begin{array}{lcl}
\varphi\in\logic{Cl} & \iff & \bm{s}'\varphi \in L. \\
\end{array}
\eqno{\mbox{$({\ast}{\ast})$}}
$$
The implication $({\Rightarrow})$ in the equivalence~\mbox{$({\ast}{\ast})$} holds because $\logic{Cl}\subseteq L$ and the logic $L$ is closed under substitution.
To justify the implication $({\Leftarrow})$ in the equivalence~\mbox{$({\ast}{\ast})$}, suppose that $\varphi\not\in\logic{Cl}$. 
Let $v\colon\mathcal{P}\to\{0,1\}$ be a valuation that refutes the formula $\varphi$. Consider the Kripke frame $F' = \langle \omega \cup \{\omega\}, >\rangle$ (i.e., $F'$ is $1+\omega^\ast$ with the strict order), and take a valuation such that for each $k\in\omega$, the following condition holds:
$$
\begin{array}{lcl}
k\models p & \iff & \mbox{$v(p_k)=1$.} \\
\end{array}
$$
Then it is easy to see that for each $k\in\numN$,
$$
\begin{array}{lcl}
\omega\models \beta'_k & \iff & v(p_k) = 1, \\
\end{array}
$$
which implies that $\omega\not\models\bm{s}'\varphi$. It remains to note that $F'\models\logic{GL.3}$, so $\bm{s}'\varphi\not\in\logic{GL.3}$, and hence $\bm{s}'\varphi\not\in L$, since $L\subseteq\logic{GL.3}$.
\end{proof}

\begin{theorem}
Let $L$ be a modal logic containing $\logic{Cl}$ and contained in $\logic{wGrz}$. Then the $L$-satisfiability problem is $\ccls{NP}$-hard for the variable-free fragment of $L$.
\end{theorem}

\begin{proof}
For each $n\in\numNp$, let 
$$
\begin{array}{lcl}
\alpha''_n & = & \Diamond \Box^+\Diamond\top \wedge \delta'_n\wedge \neg\delta'_{n+1}; \\
\beta''_n & = & \Diamond (p \wedge \alpha'_n). \\
\end{array}
$$
Let $\mathcal{P}=\{p_k : k\in\numNp\}$ be set of propositional variables, and let $\bm{s}''$ be the substitution defined for each variable $r$ as follows:
$$
\begin{array}{lcl}
\bm{s}''r & = & 
 \left\{
 \begin{array}{rl}
 \beta''_k, & \mbox{if $r=p_k$;} \\
       r, & \mbox{if $r\not\in\mathcal{P}$.} 
 \end{array}
 \right.
\end{array}
$$
Then, for every classical formula $\varphi$ with variables $p_1,\ldots,p_n$, the following equivalence holds:
$$
\begin{array}{lcl}
\varphi\in\logic{Cl} & \iff & \bm{s}''\varphi \in L. \\
\end{array}
\eqno{\mbox{$({\ast}{\ast}{\ast})$}}
$$
The implication $({\Rightarrow})$ in the equivalence~\mbox{$({\ast}{\ast}{\ast})$} holds because $\logic{Cl}\subseteq L$ and $L$ is closed under substitution.
To justify the implication $({\Leftarrow})$ in the equivalence~\mbox{$({\ast}{\ast}{\ast})$}, suppose that $\varphi\not\in\logic{Cl}$. 
Let $v\colon\mathcal{P}\to\{0,1\}$ be a valuation that refutes the formula $\varphi$. 
Consider the Kripke frame $F'' = \langle W, R\rangle$, where
$$
\begin{array}{lcl}
W & = & \{w, w^\ast\}\cup \{w^k_m : \mbox{$k\in \numNp$, $0\leqslant m\leqslant k$}\}; \\
R & = & \{(w, x) : x\in W\setminus\{w\}\}\cup {}\\
  &   & \{(w^k_m, w^k_{m+1}) : \mbox{$k\in \numNp$, $0\leqslant m\leqslant k-1$}\} \cup {}\\
  &   & \{(w^k_0, w^\ast) : k\in \numNp\}\cup {} \\
  &   & \{(w, w_0^k) : v(p_k) = 1\}. \\
\end{array}
$$
Then it is easy to see that for each $k\in\numN$,
$$
\begin{array}{lcl}
w\models \beta''_k & \iff & v(p_k) = 1, \\
\end{array}
$$
which implies that $w\not\models\bm{s}''\varphi$. It remains to note that $F''\models\logic{wGrz}$, so $\bm{s}''\varphi\not\in\logic{wGrz}$, and hence $\bm{s}''\varphi\not\in L$, since $L\subseteq\logic{wGrz}$.
\end{proof}

Returning to $\logic{E}$, $\logic{EM}$, $\logic{EN}$, and $\logic{EMN}$, it is easy to see that they satisfy the conditions of all three theorems, so we conclude that even their variable-free fragments are $\ccls{coNP}$-complete. 

This result is interesting because we were unaware of any logic with a $\ccls{coNP}$-complete variable-free fragment. For instance, the variable-free fragments of logics between $\logic{K}$ and $\logic{wGrz}$ are $\ccls{PSPACE}$-hard~\cite{ChR:2003:AiML,AgadzhanianRybakov:2022:arXiv}, while the variable-free fragments of logics between $\logic{GL}$ and $\logic{GL.3}$, though containing infinitely many pairwise inequivalent formulas, are polynomially decidable (see, e.g.,~\cite{MR:2002}).

The obtained results also have other consequences; let us give a few examples. Let
$$
\begin{array}{lcl}
    \bm{a4} & = & \Box p \to \Box\Box p;\\
    \bm{ad} & = & \Diamond \top.
\end{array}
$$
The logics $\logic{E4}$, $\logic{EM4}$, $\logic{EN4}$, and $\logic{EMN4}$ are contained in $\logic{wGrz}$ and are $\ccls{coNP}$-complete~\cite{Vardi:1989}, and therefore their variable-free fragments are also $\ccls{coNP}$-complete. Moreover, the logics $\logic{ED}$, $\logic{EMD}$, $\logic{END}$, $\logic{EMND}$, $\logic{E4D}$, $\logic{EM4D}$, $\logic{EN4D}$, and $\logic{EMN4D}$ are contained in $\logic{Grz.3}$ and are $\ccls{coNP}$-complete~\cite{Vardi:1989}, and thus their two-variable fragments are also $\ccls{coNP}$-complete. However, using the methods described in~\cite{ChR:2003:AiML}, it can be proved that their one-variable fragments are already $\ccls{coNP}$-hard: similarly to the case of $\logic{wGrz}$, one can use the fact that these logics are contained in $\logic{Grz}$ and then apply the one-variable formulas proposed in~\cite{ChR:2003:AiML} for this logic.


\bigskip

\textit{The first author is a winner of the ``Junior Leader'' competition held by the ``BASIS'' Foundation for the Development of Theoretical Physics and Mathematics.}

\end{document}